\documentclass{amsart}
\usepackage{amsmath,amsthm,amssymb}
\usepackage{graphicx}
\usepackage{tikz-cd}
\usepackage{paralist}
\usepackage{environ}
\usepackage{xcolor}
\usepackage{hyperref}
\usepackage{centernot}
\usepackage{mathtools}
\usepackage{marvosym}
\usepackage{xcolor,cancel}
\usepackage{pstricks}
\usepackage{caption}

\makeatletter
\def\square{\pst@object{square}}
\def\square@i(#1,#2)#3{{\use@par\solid@star\psframe[origin={#1,#2}](#3,#3)}}
\makeatother

\makeatletter
\DeclareFontFamily{U}{tipa}{}
\DeclareFontShape{U}{tipa}{bx}{n}{<->tipabx10}{}
\newcommand{\arc@char}{{\usefont{U}{tipa}{bx}{n}\symbol{62}}}%
\newcommand{\arc}[1]{\mathpalette\arc@arc{#1}}
\newcommand{\arc@arc}[2]{%
  \sbox0{$\m@th#1#2$}%
  \vbox{
    \hbox{\resizebox{\wd0}{\height}{\arc@char}}
    \nointerlineskip
    \box0
  }%
}
\makeatother

\makeatletter
\newcommand{\doublewedge}{\big@doubleop{\wedge}}
\newcommand{\big@doubleop}[1]{%
  \DOTSB\mathop{\mathpalette\big@doubleop@aux{#1}}\slimits@
}

\newcommand\big@doubleop@aux[2]{%
  \sbox\z@{$\m@th#1#2$}%
  \makebox[1.35\wd\z@][s]{$\m@th#1#2\hss#2$}%
}
\newcommand{\abs}[1]{\left|#1\right|}     
\newcommand{\Int}{\mbox{int}} 
\newcommand{\cyc}{\mbox{cyc}} 

\newcommand{\near}{\delta} 
\newcommand{\dnear}{\delta_{\Phi}} 
\newcommand{\tnear}{\mathop{\delta}\limits_t} 
\newcommand{\mnear}{\mathop{\delta}\limits_{\varepsilon}} 
\newcommand{\knear}{\mathop{\delta}\limits_{\kappa}} 

\usepackage{subcaption}

\makeatletter
\renewcommand{\p@subfigure}{}
\makeatother

\theoremstyle{plain}
\newtheorem{theorem}{Theorem}
\newtheorem{lemma}{Lemma}
\newtheorem{proposition}{Proposition}
\newtheorem{remark}{Remark}
\newtheorem{definition}{Definition}

\newtheorem{example}{Example}

\newtheorem{axiom}{Axiom}
\usepackage[title]{appendix}

\begin{document}

\title{Connectedness, Continuity, and Proximities in Temporal Digital Topology}

\author[J.F. Peters]{James Francis Peters}
\address{
Department of Electrical \& Computer Engineering,
University of Manitoba, Winnipeg, MB, R3T 5V6, Canada and
Department of Mathematics, Adıyaman University, Adıyaman, Türkiye.
}
\email{james.peters3@umanitoba.ca}
\thanks{The research has been supported by the Natural Sciences \&
Engineering Research Council of Canada (NSERC) discovery grant 185986 
and Instituto Nazionale di Alta Matematica (INdAM) Francesco Severi, Gruppo Nazionale per le Strutture Algebriche, Geometriche e Loro Applicazioni grant 9 920160 000362, n.prot U 2016/000036 and Scientific and Technological Research Council of Turkey (T\"{U}B\.{I}TAK) Scientific Human
Resources Development (BIDEB) under grant no: 2221-1059B211301223.}
\author[T. Vergili]{Tane Vergili}
\address{
Department of Mathematics, Karadeniz Technical University, Trabzon Türkiye.
}
\email{tane.vergili@ktu.edu.tr}


\date{}

\begin{abstract}
This paper introduces the structure and axioms for a temporal digital topology (TDT) with the focus on digital connectedness, continuity and proximities in TDT spaces.  Results are given for temporal digital adjacencies, connectedness and proximities that occur in time-constrained digital
topological spaces.  The application for this work is a proximity space view of elapsed times present in every sequence of video frames.
\end{abstract}
\maketitle
\tableofcontents

\section{Introduction}
This paper carries forward and extends recent work on digital topology (e.g.,~\cite{Bleile:digitaImageConstructions},\\
~\cite{BoxerStaecker2016}, especially, ~\cite{Vergili2020DigitalHausdorff}), with important advances beyond the original work on adjacency in digital images by A. Rosenfeld~\cite{Rosenfeld:adjacency}. These advances reflect recent work on nearness in proximity spaces (see, e.g.,~\cite{IsKaraca2023ProximalHomotopy},
~\cite{HaiderPeters2021temporalProximities},~\cite{PetersVergili2023goodCoverings}). In this paper, we call attention to the contrast between the view of digital images in terms of the image geometry limited to the location and orientation of image picture elements (pixels) and a view of picture elements called voxels in time-constrained digital images called frames in videos in which results are given not only in terms of voxel location but also relative to the more definitive voxel value.  The combination of voxel location and voxel value provides a precise basis for distinguishing between video frame proximities, both spatially and descriptively.

Briefly, in a digital topology space $X$, a planar digital image $Img$ is a mapping $Img: X\to \mathbb{R}$ whose domain $X$ is a collection of lattice points in which each picture element $p$ has location $\mathbb{Q}\times \mathbb{Q}$ and value $Img(p)\in \mathbb{R}$ as in~\cite{Bleile:digitaImageConstructions}, instead of $\mathbb{Z}\times \mathbb{Z}$ as in~\cite{Boxer1999}. This form of picture element location makes it possible to consider pixels with integer coordinates as well as points with fractional coordinates in subpixels that are located between pixels. In addition, the usual form of adjacency~\cite{Boxer1999, Klette2004, Rosenfeld1979}, now includes column, row, point and subimage adjacency in addition to 4-adjacency and 8-adjacency.  The space $X$ is Hausdorff, {\em i.e.}, distinct picture elements in a digital image reside in disjoint neighborhoods. 

There are two distinct type digital images, namely, an image (denoted by $Img_t$ or simply by ${fr_t}$ which is in a time-ordered sequence of images called frames in which each frame occurs at an elapsed time $t$ in a video and single images (denoted by $Img$) that are not frames a video.  A picture element  $p \in X$ with integer coordinates in a single image is called a {\bf pixel} and a {\bf voxel} in video frame.

\begin{definition}\label{def:pixelValue}{\rm {\bf [Pixel Value]}}\\
	Let $X$ be a digital topological space, $p$ a picture element in an image.  Every image  $Img$ is a mapping, {\em i.e.},
	\begin{align*}
		Img:X &\to \mathbb{R},\ \mbox{a digital image with}\\
		Img(p) &\in \mathbb{R}\ \left(\mbox{\bf pixel value} \right). 
	\end{align*}
\end{definition} 

\begin{example}
	Sample color, grayscale and binary video frames are shown in Figure~\ref{fig:frames}. For example, the cylinder in Figure~\ref{fig:fr1} has a green interior and magenta exterior.  By contrast, frames $fr_2, fr_3$ in Figure~\ref{fig:frames} are examples of grayscale and binary frames (no color).  The frame in Figure~\ref{fig:fr2} displays a grayscale cylinder in which the interior and exterior of the cylinder are different shades of gray.  And the frame in  Figure~\ref{fig:fr3} displays a binary cylinder in which the interior is black and exterior of the cylinder is white. 
	For any grayscale frame $fr$ containing single picture element $p$ (called a voxel), we have $fr(p)\in [0,1]$ with 
	\begin{align*}
		fr(p) &= 0,\ \mbox{black}\\
		fr(p) &= 0.5,\ \mbox{median gray}\\
		fr(p) &= 1.
	\end{align*} 
\end{example}

\begin{remark}
	For simplicity and in keeping with the approach in~\cite{Bleile:digitaImageConstructions}, we assume that video frames are grayscale in our illustrations.
\end{remark}

\begin{figure}
	\centering
	\begin{subfigure}[b]{0.3\textwidth}
	\centering
	\includegraphics[width=\textwidth]{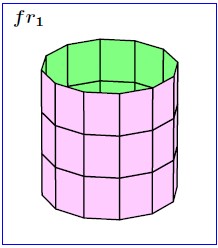}
	\caption{Color frame $fr_1$}
	\label{fig:fr1}
	\end{subfigure}
	\hfill
	\begin{subfigure}[b]{0.3\textwidth}
	\centering
	\includegraphics[width=\textwidth]{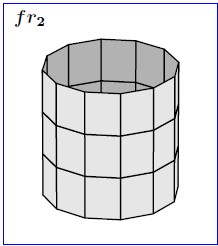}
	\caption{Greyscale $fr_2$}
	\label{fig:fr2}
	\end{subfigure}
	\hfill
	\begin{subfigure}[b]{0.3\textwidth}
	\centering
	\includegraphics[width=\textwidth]{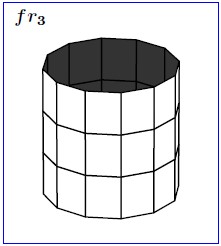}
	\caption{Binary frame $fr_3$}
	\label{fig:fr3}
	\end{subfigure}
	\caption{Color, Greyscale \& binary video frames}
	\label{fig:frames}
\end{figure}

\begin{figure}
	\centering
	\begin{subfigure}[b]{1.0\textwidth}
		\centering
		\includegraphics[width=\textwidth]{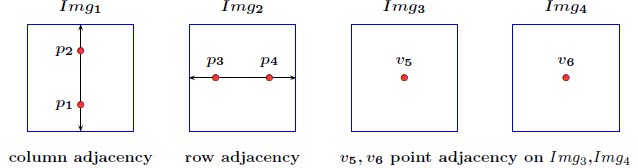}
		\caption{ Location Adjacency in separate digital images}
		\label{fig:pointwiseAdjacency}
	\end{subfigure}
	\begin{subfigure}[b]{1.0\textwidth}
		\centering
		\includegraphics[width=\textwidth]{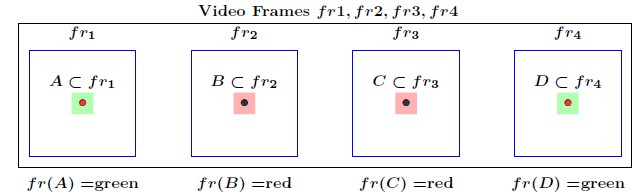}
		\caption{ Value adjacency across video frames}
		\label{fig:subimageAdjacency}
	\end{subfigure}
	\caption{$\kappa$-adjacency: {\bf column} and {\bf row} adjacency on single digital images $Img_1, Img_2$, {\bf point Location-adjacency} across images $Img_3, Img_4$ and {\bf Location-Value-adjacency} across video frames $fr_1, fr_2, fr_3, fr_4$ with $A \ \near\ D, B\ \near \ C$.}
	\label{fig:kAdjacency}
\end{figure}

\begin{remark}
	The usual picture element is in a set of lattice points in $\mathbb{Z}\times\mathbb{Z}$ in a Euclidean 2-dimensional space (see, {\em e.g.},\cite[\S 2.1]{Boxer1999}).  In Definition~\ref{def:pixelValue}, a picture element $p\in \mathbb{Q}\times \mathbb{Q}$, which can either be a pixel or a sub-pixel, {\em i.e.}, point between pixels.  On a micrometer scale, there are vast image regions between voxels that are under-represented by voxels that we see in a video frame. Also from Definition~\ref{def:pixelValue}, a digital image $Img$ maps each picture element to a real value. The distinction between location and value of a picture element in a digital image facilitates a persistent homology of digital images viewed as simplicial complexes, {\em e.g.}, intervals of persistence such as those in B. Bleile {\em et. al.}\cite[\S 6, p. 24]{Bleile:digitaImageConstructions}.
\end{remark} 

\begin{remark}
	The gray box containing a red dot representing a voxel $frE(x,y)$ in video frame $frE$ in Figure~\ref{fig:videoFrame} indicates a region surrounding $frE(x,y)$, where a possible sub-voxel is located at $frE(x+\frac{1}{2},y+\frac{1}{2})$.  A {\bf sub-voxel} is a voxel present in between voxels at integer locations such as location $(x,y)\in \mathbb{Z}\times\mathbb{Z}$ in the frame in Figure~\ref{fig:videoFrame}.
\end{remark}

\begin{definition}\label{def:video}{\rm {\bf [Video]}}\\
	In a space $X$, a {\bf video} is a collection of time-ordered subsets $2^X$ called {\bf frames}.  Let $fr E \in 2^X$ be a frame.  A picture element  in $fr E$ is called a {\bf (sub)voxel}, where each (sub)voxel $fr E(x,y,t)\in fr E$ has coordinates $x,y,t$ that include an elapsed-time component $t$ as well as horizonal and vertical components $x,y$.
\end{definition}

\begin{example}{\bf [Video Frame and its Sub-Frame]}\\
	A video frame $fr E$ and a subframe $fr A\subset  fr E$ are shown in Figure~\ref{fig:videoFrame} containing a voxel $fr A(x,y)$ with spatial coordinates $x,y$.  The red dot {\textcolor{red}{$\boldsymbol{\bullet}$}} indicates the centroid of frame $fr A$.
\end{example}

\begin{figure}
\centering 
\scalebox{0.8}{\includegraphics{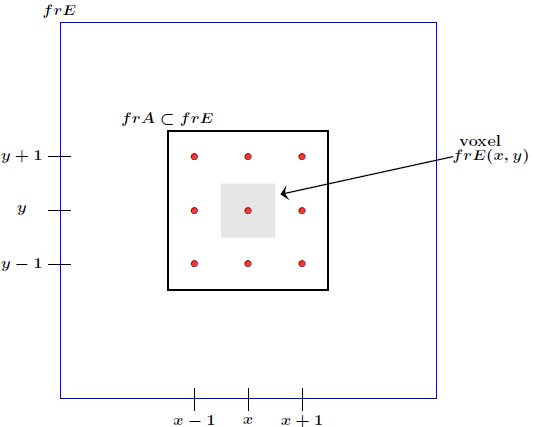}}
\caption{Voxel in a video frame}
\label{fig:videoFrame}
\end{figure}

\begin{remark} {\bf [Single Digital Image vs. Video Frame]}\\
	Unlike a single digital image, each frame $fr_t$ in a video $E$ in a space $X$ is a collection of two kinds of distinct subsets, namely, background and foreground,
	which are defined relative to a subsequent frame $fr_{t+1}\in E$.
\end{remark} 

Let $fr A$ be a $3\times 3$ subframe with center voxel at $fr A(x,y)$ row-adjacent to  $fr A(x+1,y)$ column-adacent to $fr A(x,y+1)$ in a video frame $fr E$ in a space $X$. The timeless digital partial derivative $\frac{\partial fr A(x,y)}{\partial x}$ of a voxel $fr A(x,y)$ in the horizontal direction is defined by
\[
\frac{\partial fr A(x,y)}{\partial x} = fr A(x+1,y) - fr A(x,y). 
\]
Similarly, the partial derivative $\frac{\partial fr A(x,y)}{\partial y}$ in the vertical direction is defined by
\[
\frac{\partial fr A(x,y)}{\partial y} = fr A(x,y+1) - fr A(x,y).
\]

\noindent This form of digital partial derivative appears in~\cite[\S 4.5.1, p. 98]{SolomonBreckon2011} (see, also, ~\cite[\S 3.3.1, p. 95]{Aubert2006heat}).

Time-ordering of subsets in a space $X$ occurs in the case where each subset in $X$ appears after an elapsed time.  

\begin{definition}\label{def:videoForeground}{\rm\bf [Video Frame Background and Foreground]}\\
	In a space $X$, a {\bf video} is a collection of time-ordered subsets in $2^X$ called {\bf frames}.  Let  $fr_t E, fr_{t+1} E\in 2^X$ be temporally adjacent video frames in $X$, {\em i.e.}, let $t,t+1$ be the elapsed times of frames $fr_t E, fr_{t+1} E$.  The {\bf background} of $fr_t E$ is the collection of all voxels $fr_{t} E(x,y,t)$ at location $x,y$ at time $t$  such that
	\[
	\abs{\frac{\partial fr_{t+1} E(x,y,t+1)}{\partial x}-\frac{\partial fr_{t} E(x,y,t)}{\partial t}} = 0,\ \mbox{for all}\ x,y.
	\] 
	and
	\[
	\abs{\frac{\partial fr_{t+1} E(x,y,t+1)}{\partial y}-\frac{\partial fr_{t} E(x,y,t)}{\partial y}} = 0.
	\]
	That is, if we compare temporally adjacent frames at elapsed times $t, t+1$, corresponding voxel lumens\footnote{The unit of measurement of the value of a voxel $fr_t E(x,y,t)$ is {\bf lumens}  (brightness), which is also used to classify light bulbs.}  $fr_t E(x,y,t), fr_{t+1} E(x,y,t+1)$ 
	In other words, voxels $fr_t E(x,y,t),fr_{t+1} E(x,y,t+1)$ belong to the background of a video frame at time $t$, provided there is no change in the partial derivatives of $fr_t E, fr_{t+1} E$.   Similarly, voxels $fr_t E(x,y,t),fr_{t+1} E(x,y,t+1)$ belong  to the {\bf foreground} of a video frame, provided there is a rate-of-change in corresponding voxels in temporally adjacent frames, i.e.,
	\[
	\abs{\frac{\partial fr_{t+1} E(x,y,t+1)}{\partial x}-\frac{\partial fr_{t} E(x,y,t)}{\partial x}} > 0,
	\] 
	or
	\[
	\abs{\frac{\partial fr_{t+1} E(x,y,t+1)}{\partial y}-\frac{\partial fr_{t} E(x,y,t)}{\partial y}} > 0.
	\] 
\end{definition}

\begin{remark}{\bf [What a video frame foreground tells us]}\\
	Interest in the foreground of a video frame stems from the fact that changes in voxel lumens values in video frames are associated with motion that has been recorded in a video.  Tracking video frame voxel changes leads to persistence diagrams which plot the appearance, disappearance and possible reappearance of changing foreground voxel values over periods of time. 
\end{remark}

\begin{figure}
	\centering 
	\scalebox{0.8}{\includegraphics{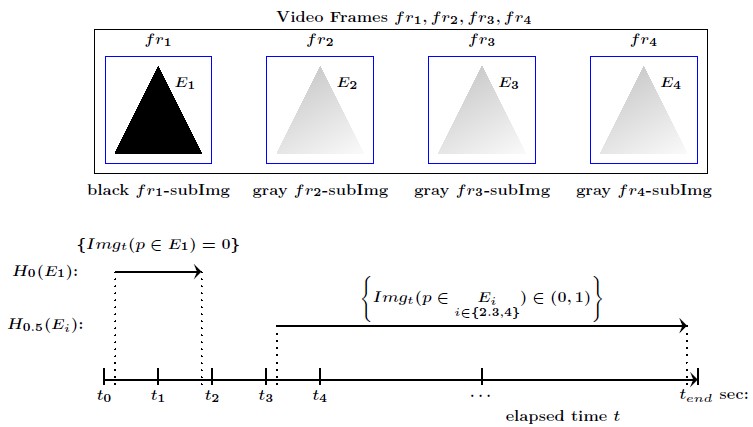}}
	\caption{Video frame shape elapsed time persistence diagram}
	\label{fig:framePersistenceDiagram}
\end{figure}

\begin{example} {\bf [Temporal Intervals of Persistence Diagram]}\\
	Let $p \in X$ be a pixel and $H_{i\in Z}:2^X\to 2^{\mathbb{R}^2}$ defined by
	\begin{align*}
		fr\in 2^X  &=\ \mbox{video frame}.\\
		Img_t &= fr_t\ \forall t\in [t_0,t_{end}].\\
		E &\subset fr\ \mbox{bounded frame region}.\\
		p &\in E.\\
		H_0(E) &=  \left\{t \ : \ Img_t(p)=0, \ \forall p\in E\right\}     
		\  \mbox{(lifespan of black pixels)}. \\
		H_{0.5}(E) &=     \left\{t \ : \ Img_t(p), \in (0,1)  \ \forall p\in E\right\}     \ \mbox{(lifespan of gray pixels)}.\\
		H_1(E) &=   \left\{t \ : \ Img_t(p) = 1, \ \forall p\in E\right\} \ \mbox{(lifespan of white pixels)}.
	\end{align*}
	Consider a moving object appears as a bounded foreground region recorded in a video frame $fr_t$. Then this moving object - as a bounded region -is partitioned into (bounded) subregions. For subregions that are pairwise distinct, the closures of two subregions may intersect along their boundaries, and each subregion is $\knear$ connected  and $fr_t(p)$ is fixed  for all voxels $p$ in a subregion. The collection of $\knear$-connected (Definition~\ref{def:knearconnected}) subregions will cover the moving object so that we could define the "cat" number of a moving object. The Cat number would be the minimum number of subregions that cover the moving object. 
	Also, given an anonymous frame $fr_t$ on $X$, if two distinct pixels/voxels $p$ and $q$ are $\knear$-adjacent and $fr_t(p)=fr_t(q)$, then these two pixels/voxels are a strongly connected pair.
For simplicity, a sample video frame shape elapsed time persistence diagram is shown in Fig.~\ref{fig:framePersistenceDiagram} relative to a black and several gray triangle shapes in a sequence of video frames.  This diagram indicates that\newline \vspace*{-0.49cm}
\begin{figure}[h]
	\centering 
	\scalebox{0.8}{\includegraphics{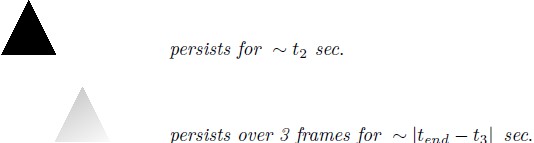}}
\end{figure}

\noindent In other words, a black $\color{black}{\boldsymbol{\bigtriangleup}}$ appears only in frame $fr_1$, persisting for approximately $t_2$ sec. and then disappears.  By contrast, a gray $\bigtriangleup$ first appears in frame $fr_2$ and then reappears in frames $fr_3$ and $fr_4$, persisting for approximately $\abs{t_{end}-t_3}$ sec.
\end{example} 

\begin{definition}\label{def:voxel}{\rm {\bf [Voxel Location]}}\\
	A planar {\bf voxel} is a picture element in a video frame $fr_t$ is a time-constrained digital image.  That is, every planar video frame is a digital image in which each voxel $p$ has a location $(r,c,t)\in \mathbb{Z}\times\mathbb{Z}\times\mathbb{R}_{\geq 0}$.  Every voxel also has an elapsed time coordinate $t$.
\end{definition}

\begin{definition}\label{def:voxel}{\rm {\bf [Temporal digital topological space]}}\\
	A temporal digital topological (TDT) space is a set $X\subset \mathbb{Q}\times\mathbb{Q}\times\mathbb{R}$ with
\begin{align*}
fr E &\in 2^X,\ (\mbox{{\bf Frame $E$ in TDT space $X$}}) \\
p &= (x,y,t)\in X,\ \mbox{a picture element at location}\ (x,y)\ \mbox{at time}\ t,\\
p &\in \mathbb{Q}\times\mathbb{Q}\times\mathbb{R},\\
fr E &= \left\{p(x',y',t): x-0.5\leq x'\leq x+0.5,y-0.5\leq y'\leq y+0.5, t\geq 0 \right\}.
\end{align*} 
\end{definition}

Given a vovel $p=(x,y,t)$ in a frame $fr E$ in a TDT space $X$, let $bdy(p)$ be the set of points consisting of all subpixels in the boundary of $p$, i.e.,
\[ bdy(p)=\{ (r,s,t) \in \mathbb{Q}\times \mathbb{Q}\times \mathbb{R} \ | \ r=x\pm 0.5 \ \mbox{and} \ s=y\pm 0.5\}.\]

\begin{definition}\label{def:voxel}{\rm {\bf [Adjacent Picture Elements]}}\\
	Let $bdy(p)$ denote the boundary of a voxel $p\in fr E$ in a TDT space $X$.  Two picture elements $p,q$ are adjacent if and only if
	\[
	bdy(p) \cap bdy(q) \neq 0.
	\]
\end{definition}

\begin{definition}\label{def:voxelValue}{\rm {\bf [Voxel Value]}}\\
	A planar video frame $fr_t E\in 2^X$ (a time-constrained subset $fr_t E$ in a TDT space $X$) at time $t$ containing a voxel $p$ is a mapping $fr_t E: X\to \mathbb{R}$ with a frame value $fr_t E(p)\in \mathbb{R}$.  
\end{definition} 

\section{Digital Topology Axioms}  

We have the following axiom for digital images. For digital images, we have the following axioms. 

\begin{axiom}\label{axiom:digitalTopologySpace}{\rm {\bf [Planar Digital Topology Space}]}\\
	A {\bf digital topology space} is a Hausdorff space $X\subset \mathbb{Q}\times \mathbb{Q}$ containing digital images  $Img:X\to \mathbb{R}$.
\end{axiom}

\begin{axiom}\label{axiom:pictureElement}{\rm {\bf [Digital Picture and sub-Picture Elements]}}\\
	A digital image {\bf pixel} has location $(x,y)$ with $p\in \mathbb{Z}\times\mathbb{Z}$.  A video frame {\bf voxel} with elapased time $t$ has location $(x,y,t)$ with $p\in \mathbb{Z}\times\mathbb{Z}\times\mathbb{R}$.  A picture element between pixels is called a {\bf sub-pixel} at location $(r,c)\in \mathbb{Q}\times\mathbb{Q}$.  A picture element between voxels is called a {\bf sub-voxel} at location $(r,c,t)\in \mathbb{Q}\times\mathbb{Q}\times\mathbb{R}$.
\end{axiom}

\begin{axiom}\label{axiom:nonemptySubimage}{\rm {\bf [Digital Subimage (Subvoxel)]}}\\
	For $A\subset X$, a subimage  of an image $Img: X\to\mathbb{R}$ is a mapping $\left.Img\right|_A:A\to \mathbb{R}$.
	Every subimage in a digital image is nonempty.
\end{axiom}

\begin{axiom}\label{axiom:frame}{\rm {\bf [Voxel Value]}}.\\
	Every planar frame $fr_t E$ in a video is a time-constrained digital image in TDT (temporal digital topology) space $X\subset \mathbb{Q}\times \mathbb{Q}\times \mathbb{R}$. 
	A picture element in $fr_t E$ is called a ({\bf voxel} i.e., volume picture elements), since every voxel $p$ in $fr_t E$ at location $(x,y,t)$ has value $fr_t(p)\in\mathbb{R}$ with an elapsed time $t$. For simplicity, we write $fr_t E = Img_t : X \to \mathbb{R}$
\end{axiom}

\begin{remark}
	Since the focus here is on picture elements in digital images that are video frames, we write {\bf voxel}, instead of {\bf pixel}, a picture element in a single image that is not a video frame.   For simplicity, we usually write $fr_t$ instead of $fr_t E$ and $fr$ instead of $fr E$.
\end{remark}

\begin{figure}
	\centering 
	\scalebox{0.8}{\includegraphics{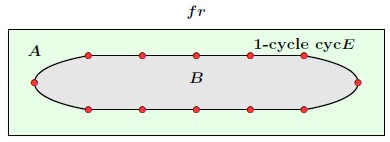}}
	\caption{1-cycle $E$ with interior $\Int(E)=B$ in video frame $fr$ with region $A=fr\setminus \cyc E$.}
	\label{fig:1-cycle2}
\end{figure}

\begin{axiom}\label{axiom:JordanCurve}{\rm {\bf [1-Cycle Simple Closed Curve]}}\\
	In a video frame, an {\bf edge} is a line segment with each end attached to a voxel.  A {\bf 1-cycle} is a sequence of edges (each with a common voxel) with no self-loops forming a simple closed curve with nonempty interior. 
\end{axiom}

\begin{example}
	A sample 1-cycle $\cyc E$ in a video frame $fr$ is shown in Figure~\ref{fig:1-cycle2}.
\end{example}

\begin{lemma}\label{lemma:nonemptyImage}
	A digital image is not an empty set.
\end{lemma}
\begin{proof}
	From Axiom~\ref{axiom:nonemptySubimage}, every subimage $subImg$ is nonempty.  Hence, $Img\neq \emptyset$.
\end{proof}

\begin{proposition}
	Every video frame is nonempty.
\end{proposition}
\begin{proof}
	From Axiom~\ref{axiom:frame}, a video frame is a digital image.
	Hence, from Lemma~\ref{lemma:nonemptyImage}, every video frame in nonempty.
\end{proof}

\begin{theorem}\label{theorem:Jordan}{\rm {{\bf [Jordan Curve Theorem]}}}
	Every simple, closed curve partitions on planar region partitions the region into disjoint subregions.
\end{theorem}

\begin{lemma}\label{lemma:digitalJordan}
	Every digital image with a subimage that has a simple, closed curve on its boundary satisfies the Jordan Curve Theorem.  
\end{lemma}
\begin{proof}
	Let $subImg$ be a subimage in a planar digital image $Img$ with edges attached between pixels on its boundary.  From Axiom~\ref{axiom:JordanCurve}, $subImg$ is bounded by a simple, closed curve.  Hence, from Theorem~\ref{theorem:Jordan}, $subImg$ partitions $Img$ into disjoint subregions $A,B$ with $A$ exterior to $subImg$ and $B$ in the interior of $subImg$.
\end{proof}

\begin{example}
	A sample 1-cycle in a video frame $fr$ is shown in Figure~\ref{fig:1-cycle2}.
\end{example}

\begin{theorem}
	Every subimage in a video frame satisfies the Jordan Curve Theorem.
\end{theorem}
\begin{proof}
	Every video frame $fr_t$ is a digital image.  Let $A$ be a subimage in a video frame.  From Axiom~\ref{axiom:JordanCurve}, $A$ is a simple, closed curve. Hence, from Lemma~\ref{lemma:digitalJordan}, $A$ satisfies the Jordan Curve Theorem.
\end{proof}

Spatially near subsets in a digital image ({\em i.e.}, subsets that share points) reside in a discrete proximity space.

\begin{definition}\label{def:discreteProximity} {\rm {\bf [Discrete Proximity]}} \\
	For any pair of nonempty subimages $A, B$ in a digital image $(X, Img)$, $A$ is near $B$ (denoted by $ A\ \near_d \ B$), provided $A$ and $B$ have points in common, {\em i.e.},  
	\[
	A\ \near_d \  B\ \mbox{iff}\ Img A\cap Img B\neq \emptyset.
	\]  
	Hence, $\delta_d$ is a discrete proximity relation and $(Img,\delta_d)$ is a discrete proximity space~\cite[\S 40.2, pp. 266-267]{Willard1970},~\cite[\S 1,p. 9]{Naimpally70}.
\end{definition}

The following forms of adjacencies for a pair of pixels $p$ and $q$ in $Img$ analogous to the $\kappa$-adjacency in digital topology~\cite{Boxer1999}. \\

\begin{compactenum}[1$^o$]
	\item [{\bf column adjacency}.] Two pixels $p=(x,y)$ and $q=(x',y')$ are column adjacent, provided $y=y'\pm 1$. 
	\item [{\bf row adjacency}.] Two pixels $p=(x,y)$ and $q=(x',y')$ are row adjacent, provided $x=x'\pm 1$. 
	\item [{\bf diagonal adjacency}.] Two pixels $p$ and $q$ are diagonal adjacent, provided they are both column and row adjacent.  \\
\end{compactenum}

Here, the column and row adjacency correspond to $4$-adjacency  and the diagonal adjacency corresponds to $8$-adjacency in a planar digital image~\cite{Boxer1999}.

\begin{definition} \label{def:adjacent} {\rm {[\bf Adjacent subimages}]} \\
	Given a digital image $(X, Img)$,  two subimages  $A$ and $B$, are {\bf adjacent}, denoted by $A \ \knear\ B$, provided there exist pixels $p\in A$ and $q \in B$ such that $p=q$ or $p$ and $q$ are adjacent. 
\end{definition}

\begin{remark}
	Notice that discrete proximity implies adjacency, i.e., $A \ \near_d \ B$ implies $A \ \knear \ B$.
\end{remark}

\begin{definition}  \label{def:knearconnected} {\rm {\bf [$\knear$-connectedness]}} \\
	A bounded subimage $E$ with a non-empty interior is said to be  {\bf $\knear$-connected}, provided for each pair of distinct pixels $p$ and $q$ in $E$, there exists a finite sequence of pixels $p=v_0, p_1, \dotsc, p_n=q$ such that two consecutive pixels $p_i$ and $p_{i+1}$ are adjacent for all $i=0,1,\dotsc,n-1$. 
\end{definition}

\begin{definition} \label{def:kappacont} {\rm {\bf [$\kappa$-continuity]}} \\
	Let $X$ and $Y$ be two digital images. We say that a function $f: X \to Y$ is $\kappa$-continuous, provided the images of adjacent subimages in $X$ are also adjacent in $Y$, i.e., $A \knear B$ implies $f(A) \knear f(B)$.  
\end{definition}

\begin{proposition}
	$\knear$-connected regions are preserved under a continuous digital functions.
\end{proposition}
\begin{proof} 
	Let $E$ be a $\knear$-connected subregion in a digital topological space $X$ and $f$ be a $\kappa$-continuous function from $X$ to a digital topological space $Y$. For a pair of distinct pixels $p'$ and $q'$ in $f(E)$, there exist pixels $p,q$ in $E$ such that $f(p)=p'$ and $f(q)=q'$. Since $E$ is $\knear$-connected, there exist a finite sequence of pixels $p=p_0, p_1, \dotsc, p_{n-1}, p_n=q$ and each pair of pixels $p_i$ and $p_{i+1}$ are adjacent.In that case, the sequence $p'=f(p), f(p_1), \dotsc, f(p_{n-1}), f(p_n)=q'$ is also $\knear$-connected subset of $f(E)$. 
\end{proof}

Recall that the first two components of a voxel  $v=(x,y,t)$ in a video frame $fr_t: X \to \mathbb{R}$ show the location of $v$ in that frame and the last component $t$ represents the time parameter. Then we have the following forms of adjacencies for a pair of voxels from  distict video frames. \\
	
\begin{compactenum}[1$^o$]
  \item [{\bf Point-across-images adjacency}.]  Voxels $v=(x,y,t)$,  $\omega=(x,y,t')$ in a pair of separate video frames $fr_{t}$ and $fr_{t'}$ in the same location are point-across-adjacent in separate images.
  \item [{\bf Video voxel value adjacency}.]
	Voxels $v=(x,y,t)$,  $\omega=(x',y',t')$ in a pair of separate video frames $fr_{t}$ and $fr_{t'}$ are voxel value adjacent, provided $ft_{t}(v)=fr_{t+1}(\omega)$.\\
\end{compactenum}

We also have the following forms of adjacencies for video frames and their subimages.\\
	
\begin{compactenum}[1$^o$]
	\item [{\bf Video frame value adjacency}.]
	Video frames $fr_{t},fr_{t'}$ in which all picture elements have the same or one or more similar or identical signal values ({\em e.g.}, color, brightness level, gradient) are frame-value-adjacent.
	\item [{\bf Video frame subimage location-value adjacency}.] Let $A, B \in 2^X$ and $fr_{t}$ and $fr_{t'}$ two video frames on $X$. We say that the subimages   $fr_{t}A$ and $fr_{t'}B$  with voxels $v\in fr_{t}A$, $\omega \in fr_{t'}B$ are location-value-adjacent, provided  $fr_{t}A(v)=fr_{t'}B(\omega)$.  
\end{compactenum}

\begin{example}
In Figure~\ref{fig:kAdjacency}, we have
\begin{compactenum}[1$^o$]
\item $p_1,p_2$ in Figure~\ref{fig:pointwiseAdjacency} are column adjacent, since $p_1,p_2$ are in the same column in image $Img_1$.
\item $p_3,p_4$ in Figure~\ref{fig:pointwiseAdjacency} are row adjacent, since $p_3, p_4$ are in the same row in image $Img_2$.
\item $v_5\in Img_3$, $v_6\in Img_4$ in Figure~\ref{fig:pointwiseAdjacency} are point-across adjacent, since $v_5,v_6$ are in the same location in images $Img_3,Img_4$.
\item Subsets $A\subset fr_1, D\subset fr_4$ in Figure~\ref{fig:subimageAdjacency} are value-adjacent, since $fr_1 A(v) = fr_4 D(v) =$ green and also location-adjacent, since all voxels in $A,D$ are location-adjacent.  Similarly, subsets $B\subset fr_2, C\subset fr_3$ are value-adjacent, since $fr_2 B(v) = fr_3 C(v) =$ red and also location-adjacent, since all voxels in $B,C$ are location-adjacent.
\end{compactenum}
\end{example}

\begin{proposition}
Let subsets $A,B$ in video frames $fr,fr'$ be value-adjacent.  Then $A \ \dnear\ B$, i.e., $A$ is descriptively near $B$.
\end{proposition}

\begin{figure}
	\centering 
	\scalebox{0.85}{\includegraphics{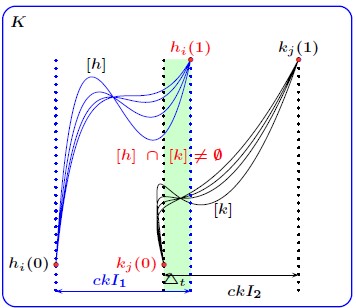}}
	\caption{Spatially far homotopy classes $[h],[k]$ such that $[h]\cap [k]=\emptyset$, temporally close with $h_i(1)\in [h], k_j(0)\in [k]$ such that $[h]\ \tnear\ [k]$.}
	\label{fig:overlappingClassClockTimes}
\end{figure}

\section{Temporal Discrete Proximity}
Temporally discrete near subsets in a video are subsets that share a common region along their overlapping clocktimes. Let $\{fr_t\}_{t\in [t_0,t_1]}$ be a collection of vide frames on a digital image $X$. Given a foreground subimage $A \subset X$, let $A_t$ denotes the position of that image at time $t$. Therefore, $A_{t_0}$ represents the initial position  and $A_{t_1}$ represents the terminal position of $A$.

\begin{definition}\label{def:timeDiscreteProximity} {\rm {\bf [Temporal Discrete Proximity]}} \\
	Given a sequence of video frames $\{fr_t\}_{t_\in [t_0,t_1]}$ and two subsets $A, B \subset X$, the temporal discrete proximity, $\tnear$,  on $X$ is defined by
	\[
	A\ \tnear \ B  \ :\Leftrightarrow \  \left\{t \in [t_0,t_1] \ : \ A_t \cap B_t \neq \emptyset \right\} \neq \emptyset
	\]
	We say that two subimages $A$ and $B$ are {\bf temporally near}, provided $A\tnear B$. 
\end{definition}

Given a digital image $(X, Img)$, subimages  $A, B \subset Img X$, and a non-negative real number $\varepsilon \geq 0$,  we define the $\varepsilon$-nearness by
\[ A \ \mbox{and} \ B \ \mbox{are} \  \varepsilon\mbox{-near} \ \Leftrightarrow \ D(A, B)\leq \varepsilon
\]

Let $A,B$ be frames that occur during a temporal interval $\bigtriangleup t$.
Notice that if  $A \tnear B$, then $A$ and $B$ are $0$-near within the time period where $A_t$ and $B_t$ are overlapping.

\begin{definition}\label{def:timeMetricProximity} {\rm {\bf [Temporal Metric Proximity]}} \\
	Given a sequence of video frames $\{fr_t\}_{t_\in [t_0,t_1]}$ and two subimages $A, B \subset X$, the temporal metric proximity, $\mnear$,  on $X$ is defined by
	\[
	A\ \mnear \ B \ \Leftrightarrow \  \left\{t\in [t_0,t_1] \ : \ A_t \ \mbox{and} \ B_t \ \mbox{are} \ \varepsilon\mbox{-near} \right\} \neq \emptyset
	\]
	We say that $A$ and $B$ are {\bf  temporally $\varepsilon$-near}, provided $A \mnear B$.
\end{definition}

\begin{example}
	Recall that homotopy classes $[h],[k]$ in a space $K$ are collection of paths $h_i\in [h], h_j\in [k]$ as shown in Fig.~\ref{fig:overlappingClassClockTimes}, with
	\begin{align*}
		[n] &= \ \mbox{mod}\ n,\\
		h_0(0) & = h_i(0)\ \mbox{for all}\ i=0,...,n-1[n], n\geq 1,\\
		h_1(0) & = h_i(1)\ \mbox{for all}\ i=0,...,n-1[n], n\geq 1.\\
		k_0(0) & = k_i(0)\ \mbox{for all}\ j=0,...,n-1[n], n\geq 1,\\
		k_1(0) & = k_i(1)\ \mbox{for all}\ j=0,...,n-1[n], n\geq 1.\\
		h_1(0) &\ \tnear\ k_1(0).\\
		h_1(1) &\ \tnear\ k_1(1).
	\end{align*}
	In this example, $[h]\ \mnear \ [k]$, since $h_1\ \tnear\ k_1$.
\end{example}

\begin{definition}\label{def:temporalDigitalTopology}{\rm {\bf [Temporal Digital Topology]}}\\
	A {\bf temporal digital topology} is a digital topology containing time-constrained digital images (or voxels).
\end{definition}

\begin{definition}\label{def:temporalProxSpace}{\rm {\bf [Temporal Proximity Space]}}\\
	Let $X$ be a sequence of video frames.  A {\bf temporal proximity space} $X$ (denoted by $\left(X,\mnear\right)	$ is a collection of sub-sequences of $A,B\in 2^X$ that occur in overlapping temporal intervals such that $A\ \mnear\ B$.
\end{definition}

\begin{remark}
	Let $A,B\in 2^X$ in temporal proximity space $\left(X,\mnear\right)$ with homotopic mappings $t_A,t_B$ such that 
	\begin{align*}
	t_A, t_B:X\times I &\to 2^{\mathbb{R}},\ \mbox{defined by}\\
	t_A(E_i) & = \left\{E_i(x)\in \mathbb{R}: E_i\in 2^A, x\in E_i\right\},\\
	t_B(H_i) & = \left\{H_i(x)\in \mathbb{R}: H_i\in 2^B, x\in H_i\right\},\\
	t_A(E_i)\ \mnear\ t_B(H_i) &\ \mbox{provided}\ t_A(E_i)\cap t_B(H_i)\neq\emptyset.
	\end{align*}
\end{remark}

\begin{lemma}\label{lemma:temporalProxSpace}{\rm {\bf [Temporal Proximity]}}\\
	Let $\left(X,\mnear\right)$ be a temporal proximity space, $A,B \subseteq X$.
	$A\ \mnear\ B$ if and only if $A\ \tnear\ B$.
\end{lemma}

\begin{remark}\label{remark:temporalProximity}
	Images that occur at same time usually do not have any points in common. Discrete temporal closeness of sets in a space $X$ is a form of descriptive closeness.  That is, if we introduce a feature $\Phi:2^X\to \mathbb{R}^2$ defined by $\Phi(A\subset X) = (t_{A_0},t_{A_1})$ (the life span of $A$).  Then,  
	\begin{align*}
		A\ &\dnear\  B,\ \mbox{provided}  \ 
		\Phi(A) \cap \Phi(B) \neq \emptyset \ \mbox{and}\\			A\ &\dnear\  B \ \mbox{and} \ A_t \cap   B_t \neq \emptyset \ \mbox{iff} \ A\tnear B.
	\end{align*}
	On the other hand, we introduce $\mnear$ (temporal metric proximity) as a convenient way of pigeonholing those sets that have overlapping lifespans.  $A\ \mnear\ B$ holds even if $A_t \cap B_t = \emptyset$.
\end{remark}

\begin{lemma}\label{lemma:digitalTemporalProximity}
	Digital images captured at the same time are elapsed-time temporal metric proximal.
\end{lemma}
\begin{proof}
	Let $Img$ and $Img'$ be digital images captured at the same time.
	From Lemma~\ref{lemma:nonemptyImage}, $Img,Img'$ are nonempty.
	Hence, from Definition~\ref{def:timeDiscreteProximity}, $Img\ \mnear\ Img'$.
\end{proof}

\begin{proposition}
	Every pair of subregions in a video frame are elapsed-time temporal metric proximal.  
\end{proposition}
\begin{proof}
	Let $frA,frB$ be a pair of video frames.  From Axiom~\ref{axiom:frame}, every video frame is a time-constrained digital image.  Hence, from Lemma~\ref{lemma:digitalTemporalProximity}, $frA\ \mnear\ frB$.
\end{proof}

\begin{proposition}
	Let $A, B$ be subimages in a digital topology space $(X,\{fr_t\})$ with $\Phi(A) = \Phi(B) = (t_0,t_1)$ (common lifespan).  If $A\ \tnear\ B$, then $A\ \mnear\ B$.  
\end{proposition}
\begin{proof}
	Immediate from Remark~\ref{remark:temporalProximity}.
\end{proof}

\begin{definition}{\rm {{\bf [Temporal Adjacency]}}}\\		
	Let $A,B$ be subsets in a digital topology space $(X, \{fr_t\})$ with $\Phi(A) = \Phi(B) = (t_0,t_1)$ (common lifespan). If there exists a time period $T \subseteq [t_0,t_1]$ such that  $A_t$ and $B_t$ are adjacent for all $t\in T$, then we say that $A$ and $B$ {\bf are temporally adjacent}.
\end{definition}

\begin{proposition}
	In a digital proximity space $(X,\near)$ containing digital images $Img: X \to \mathbb{R}$,
	\begin{compactenum}[1$^o$]
		\item Every picture element in a digital image has rational coordinates.
		\item Overlapping subimages $A, B$ are discretely proximal. 
		\item Overlapping subimages $A, B$ are  adjacent.
		\item Temporally near subimages are temporally $\varepsilon$-near. 
		\item Video frames have elapsed-time sub-voxel picture elements.
	\end{compactenum}
\end{proposition}
\begin{proof}
	Let $A, B$ be subimages in a digital image $Img$ in a topological space $X$.  Then we have\\
	1$^o$: Immediate from Axiom~\ref{axiom:digitalTopologySpace}.\\
	2$^o$: By Axiom~\ref{axiom:nonemptySubimage}, subimages $A, B$ in a digital image $Img$ are nonempty. If $A\ \cap\ B\neq \emptyset$, then, from Definition~\ref{def:discreteProximity}, $A,B$ are discretely proximal.\\
	3$^o$: If $A\ \near_d \ B$, then, from Definition~\ref{def:adjacent}, $Img A, Img B$ are adjacent.\\
	4$^o$: From Axiom~\ref{axiom:frame}, video subframes $fr_t A$ and $fr_{t}B$ are digital images. Within the time period $[t_0,t_1]$ where $A_t \cap B_t \neq \emptyset$, we immediately have that $A_t$ and $B_t$ are $0$-near. Hence $A$ and $B$ are temporal metric proximal.\\
	5$^o$: Within the time period $[t_0,t_1]$ where $A_t \cap B_t \neq \emptyset$, we immediately have that $A_t$ and $B_t$ are adjacent since they have common voxels.\\
	6$^o$: From 4$^o$, a video frame $fr$ is a digital image.  From Axiom~\ref{axiom:pictureElement}, $fr$ has picture elements that are sub-voxels. From Axiom~\ref{axiom:frame}, the sub-voxels in $fr$ are time-constrained.
\end{proof}


\begin{definition}{\rm {\bf [Video-frame Connectedness]}} \\
	Given a video $\{fr_{t}\}_{[t_0,t_1]}$ on a digital image $X$. We say that a connected subimage $E$ of $X$ is {\bf video-frame connected}, provided each $E_t$ is connected for all $t\in [t_0,t_1]$ and $E_t$ and $E_{t+1}$ are adjacent for all $t\in [t_0,t_1-1]$.
\end{definition}
	
\begin{proposition}
	The image of a (temporally) video-frame connected subregion $E \subset X$ under a $\kappa$-continuous function $f: X\to X$ is also video-frame connected. 
\end{proposition}
\begin{proof}
  It follows from Definition~\ref{def:kappacont}.
\end{proof}

\begin{definition}{\rm {\bf [Video-frame temporal Connectedness]}} \\
	Given a video $\{fr_{t}\}_{[t_0,t_1]}$ on a digital image $X$. We say that a connected subimage $E$ of $X$ is {\bf temporally video-frame connected}, provided there exists $t' \in [t_0,t_1]$ such that $E$ is video-frame connected on the time interval $[t_0,t']$ and $E_t$ dispappears for  $t'<t\leq t_1$.
\end{definition}

\begin{definition} {\rm {\bf [Temporal continuity]}}\\
	Given two videos $\{fr_t\}_{[t_0,t_1]}$ and $\{fr'_t\}_{[t_0,t_1]}$ on digital images $X$ and $Y$, respectively.  We say that a function $f: X \to Y$ is {\bf temporally continuous}, provided given subimages $A, B$ in $X$, $A$ and $B$ are temporally adjacent implies  $f(A)$ and $f(B)$ are also temporally adjacent.
\end{definition}

\begin{theorem}\label{theorem:persistence}
	Let $X$ and $Y$ be temporal digital topology spaces and let $f: X\to Y$ be temporally continuous and let $A\in 2^X, B\in 2^Y$ be temporally video-frame connected subimages over a time interval $[t_0,t']$.  
	\begin{compactenum}[1$^o$]
	\item $A\tnear B$ implies $f(A)\tnear f(B)$ for every $t\in [t_0,t']$.
	\item $A,B$ disappear for every $t > t'$.
	\item $f(A),f(B)$ disappear for every $t > t'$.
	\end{compactenum}
\end{theorem}

\begin{theorem}\label{theorem:temporallyContinuous}
  Let $X$ and $Y$ be temporal digital topology spaces. Then
  \begin{compactenum}[1$^o$]
	\item A map $f: X\to Y$ is temporally continuous if and only, for every pair of $A,B\in 2^X$ such that $A\tnear B$ implies $f(A)\tnear f(B)$. 
	\item A map $f: X\to Y$ is temporally continuous if and only, for every pair of $A,B\in 2^X$ such that $A\mnear B$ implies $f(A)\mnear f(B)$ for all $\varepsilon \leq 1$.
   \end{compactenum}
\end{theorem}

\end{document}